\documentclass{article}

\usepackage[T2A]{fontenc}
\usepackage[cp1251]{inputenc}
\usepackage[english]{babel}
\usepackage{amsmath}

\DeclareMathOperator{\proj}{proj}
\DeclareMathOperator{\spanop}{span}
\usepackage{amssymb}
\usepackage{amsthm}
\usepackage{mathrsfs}

\usepackage{dan2e}

\usepackage{cmap}
\usepackage{mathtext}
\usepackage{indentfirst}
\usepackage{amsmath, amsfonts, amssymb, amsthm, mathtools}
\usepackage{icomma}
\usepackage{algorithm}
\usepackage{algpseudocode}
\usepackage{hyperref}
\hypersetup{
	colorlinks=false,
	linkcolor=blue,
	filecolor=magenta,      
	urlcolor=cyan,
	pdftitle={Overleaf Example},
	pdfpagemode=FullScreen,
}

\usepackage{multicol}			  
\usepackage{soulutf8}             

\theoremstyle{definition}
\newtheorem{definition}{Definition}
\newtheorem{theorem}{Theorem}
\newtheorem{lemma}{Lemma}
\newtheorem{corollary}{Corollary}[theorem]
\newtheorem{assumption}{Assumption}
\usepackage[colorinlistoftodos,bordercolor=blue,backgroundcolor=blue!20,linecolor=blue,textsize=scriptsize]{todonotes}

\usepackage{multirow}
\usepackage{array}
\usepackage{tabularx}
\usepackage{booktabs}
\usepackage{tablefootnote}
\usepackage{subcaption,graphicx}
\usepackage{wrapfig}

\newcommand{\cX}{{\mathcal{X}}}
\newcommand{\cY}{{\mathcal{Y}}}

\newcommand{\eps}{\varepsilon}

\begin{document}
	
\Volume{}
\Year{2023}
\Pages{}

\udk{}

\title{Min-max optimization over slowly time-varying graphs}

\author{
	Nhat Trung Nguyen \Addressmark[1]\Emailmark[1],
	Alexander Rogozin \Addressmark[1]\Emailmark[2],
	Dmitriy Metelev \Addressmark[1]\Emailmark[3],
	Alexander Gasnikov \Addressmark[1,2,3,4]\Emailmark[4]
}

\Addresstext[1]{Moscow Institute of Physics and Technology, Moscow, Russia}
\Addresstext[2]{Institute for Information Transportation Problems, Moscow, Russia}
\Addresstext[3]{Caucasus Mathematic Center of Adygh State University, Moscow, Russia}
\Addresstext[4]{Ivannikov Institute for System Programming of the Russian Academy of Sciences, Research Center for Trusted Artificial Intelligence, Moscow, Russia}

\Emailtext[1]{ngnhtrg@phystech.edu}
\Emailtext[2]{aleksandr.rogozin@phystech.edu}
\Emailtext[3]{metelev.ds@phystech.edu}
\Emailtext[4]{gasnikov@yandex.ru}

\markboth{N.T. Nguyen, A. Rogozin, D. Metelev, A. Gasnikov}{Min-max optimization over slowly time-varying graphs}

\presentedby{Presented at AiJourney 2023 conference}

\dateA{30.08.2023}
\dateB{}
\dateC{}

\alttitle{}
\altauthor{}
\altAddresstext[]{}
\altpresentedby{}
\maketitle

\doi{}

\begin{abstract}
	Distributed optimization is an important direction of research in modern optimization theory. Its applications include large scale machine learning, distributed signal processing and many others. The paper studies decentralized min-max optimization for saddle point problems. Saddle point problems arise in training adversarial networks and in robust machine learning. The focus of the work is optimization over (slowly) time-varying networks. The topology of the network changes from time to time, and the velocity of changes is limited. We show that, analogically to decentralized optimization, it is sufficient to change only two edges per iteration in order to slow down convergence to the arbitrary time-varying case. At the same time, we investigate several classes of time-varying graphs for which the communication complexity can be reduced.
\end{abstract}




\makeatletter

\section{Introduction}

This paper studies min-max optimization problems of type
\begin{equation} \label{eq:main_prob}
	\min_{x \in \mathcal{X}} \max_{y \in \mathcal{Y}}\text{ } f(x, y) := \frac{1}{M} \sum_{m = 1}^{M} f_m(x, y),
\end{equation}
where functions $f_m(x, y)$ are convex in $x$ and concave in $y$ and $\cX, \cY$ are closed convex sets. Each function $f_m(x, y)$ is held locally at some computational node. The nodes are connected to each other via a decentralized communication network. Each agent can perform local computations and exchange information with its immediate neighbors in the network. Additionally, the network is allowed to change with time. Due to some malfunctions or disturbances, the links in the network may fail or reappear from time to time. This type of networks is called \textit{time-varying graphs}.

There are numerous applications for optimization over time-varying networks~\cite{nedic2010cooperative,nedic2020distributed}. They include distributed machine learning~\cite{rabbat2004distributed,forero2010consensus,nedic2017fast}, distributed control of power systems~\cite{ram2009distributed,gan2012optimal}, vehicle control~\cite{ren2008distributed}, distributed sensing~\cite{bazerque2009distributed}.

Decentralized optimization and min-max optimization first-order methods use two types of steps: local gradient updates and inter-node communications. We consider the case when communications are done in synchronized rounds. Therefore, the complexity of the method is measured by two quantities: number of communication rounds and number of local oracle calls. These quantities depend on the problem characteristics, which include network condition number $\chi$, function condition number $L/\mu$ and desired accuracy $\eps$. Here $L$ is the Lipschitz constant of the objective gradient and $\mu$ is the strong convexity constant.

This paper is devoted to \textit{slowly time-varying graphs}. That means that only a limited number of edges is allowed to change after each communication round. We provide lower complexity bounds for the considered class of problems. Also we propose min-max optimization methods with better communication complexity for two particular classes of slowly time-varying networks.

For optimization over networks (not min-max) lower bounds are known as well as corresponding optimal algorithms. For static graphs, lower bounds were derived in \cite{scaman2017optimal} and in the same paper the optimal dual (i.e. using a dual oracle) algorithm was proposed. Optimal decentralized methods with primal oracle were proposed in \cite{kovalev2020optimal}. Considering time-varying graphs (with arbitrary changes at each iteration), lower complexity bounds were proposed in \cite{kovalev2021lower}. Optimal primal algorithm was proposed in the same paper \cite{kovalev2021lower}, and optimal dual method first appeared in \cite{kovalev2021adom}. After that, lower bounds for slowly time-varying graphs with different velocities of network changes were studied in \cite{metelev2023consensus}. In \cite{metelev2023decentralized}, it was shown that it is sufficient to change only two edges at each iteration to make communication complexity equal to arbitrary time-varying graph. The overview of lower bounds for decentralized optimization is presented in Table~\ref{table:optimization_lower_bounds} (notation $\Omega(\cdot)$ is omitted).
\begin{wrapfigure}[10]{r}{8.5cm}
	\begin{minipage}{8cm}
		\vspace{-0.3cm}
		\begin{table}[H]
			\begin{tabular}{|c|c|c|c|}
				\hline
				& static & time-var. & slowly time-var. \\ \hline
				\rule{0.00cm}{0.5cm} comm. & $\sqrt{\chi}\sqrt\frac{L}{\mu}\log\frac{1}{\eps}$ & $\chi\sqrt\frac{L}{\mu}\log\frac{1}{\eps}$ & ${\chi}\sqrt\frac{L}{\mu}\log\frac{1}{\eps}$  \\ \hline
				\rule{0.00cm}{0.5cm} oracle & $\sqrt\frac{L}{\mu}\log\frac{1}{\eps}$ & $\sqrt\frac{L}{\mu}\log\frac{1}{\eps}$ & $\sqrt\frac{L}{\mu}\log\frac{1}{\eps}$ \\ \hline
				\rule{0.00cm}{0.4cm} paper & \cite{scaman2017optimal} & \cite{kovalev2021lower} & \cite{metelev2023decentralized} \\ \hline
			\end{tabular}
			\caption{Lower bounds for optimization}\label{table:optimization_lower_bounds}
		\end{table}
	\end{minipage}
\end{wrapfigure}

The results for decentralized saddle-point problems are analogical to optimization. Lower bounds for min-max optimization over static graphs were given in \cite{beznosikov2021distributed_2}. The same paper \cite{beznosikov2021distributed_2} proposed algorithms optimal up to a logarithmic term. The case of (arbitrary) time-varying graphs was studied in \cite{beznosikov2021near} along with methods optimal up to a logarithmic factor. Optimal algorithms for sum-type variational inequalities (a generalization of saddle point problem) were proposed in \cite{kovalev2022optimal}. Finally, this paper studies lower bounds for saddle-point problems over slowly time-varying graphs (only two edge changes per iteration). The corresponding results are presented in Table~\ref{table:saddle_lower_bounds}. It is worth noting that the lower complexity bounds are same as for optimization (Table~\ref{table:optimization_lower_bounds}), except for replacing $\sqrt{L/\mu}$ by $L/\mu$.

\begin{wrapfigure}[7]{r}{8.5cm}
	\begin{minipage}{8cm}
		\vspace{-0.8cm}
		\begin{table}[H]
			\begin{tabular}{|c|c|c|c|}
				\hline
				& static & time-var. & slowly time-var. \\ \hline
				\rule{0.00cm}{0.5cm} comm. & $\sqrt{\chi}\frac{L}{\mu}\log\frac{1}{\eps}$ & $\chi\frac{L}{\mu}\log\frac{1}{\eps}$ & ${\chi}\frac{L}{\mu}\log\frac{1}{\eps}$  \\ \hline
				\rule{0.00cm}{0.5cm} oracle & $\frac{L}{\mu}\log\frac{1}{\eps}$ & $\frac{L}{\mu}\log\frac{1}{\eps}$ & $\frac{L}{\mu}\log\frac{1}{\eps}$ \\ \hline
				\rule{0.00cm}{0.4cm} paper & \cite{beznosikov2021distributed_2} & \cite{beznosikov2021near} & This paper \\ \hline
			\end{tabular}
			\caption{Lower bounds for saddle point problems}\label{table:saddle_lower_bounds}
		\end{table}
	\end{minipage}
\end{wrapfigure}

This paper has the following organization. In Section~\ref{sec:notation_and_assumptions}, we introduce the needed assumptions and notation. After that, in Section~\ref{sec:upper_bounds}, we show how to get an acceleration in communications using additional assumptions on the time-varying network. In Section~\ref{sec:lower_bounds}, we provide lower bounds for slowly time-varying networks.

\vspace{0.2cm}
\section{Notation and Assumptions}\label{sec:notation_and_assumptions}

\noindent\textbf{Smoothness and strong convexity}.

We work with the problem \eqref{eq:main_prob}, where the sets $\mathcal{X} \subseteq \mathbb{R}^{n_x}$ and $\mathcal{Y} \subseteq \mathbb{R}^{n_y}$ are closed convex sets. Additionally, we introduce the set $\mathcal{Z} = \mathcal{X} \times \mathcal{Y} \subseteq \mathbb{R}^{n_z}$, $z = (x, y)$, $n_z = n_x + n_y$, and the operator $F$:
\begin{equation*} \label{eq:operatorF}
	F_m(z) = F_m(x, y) =
	\begin{pmatrix}
		\nabla_x f_m(x, y) \\
		-\nabla_y f_m(x, y)
	\end{pmatrix},\hspace{0.2cm}
	F(z) = \frac{1}{M} \sum_{m=1}^M F_m(z).
\end{equation*}


\begin{assumption} \label{assum:func-properties}
    Let functions $f(x,y)$ and $f_m(x, y)$ satisfy following properties:
    \begin{enumerate}
        \item Function $f(x, y)$ is $L$-smooth, i.e. for all $z_1, z_2 \in \mathcal{Z}$  it holds
	\begin{equation*}
		\| F(z_1) - F(z_2) \| \leq L \| z_1 - z_2 \|.
	\end{equation*}
        \item For all $m$, $f_m(x, y)$ is $L_{\text{max}}$-smooth, i.e. for all $z_1, z_2 \in \mathcal{Z}$ it holds
	\begin{equation*}
		\| F_m(z_1) - F_m(z_2) \| \leq L_{\text{max}} \| z_1 - z_2 \|.
	\end{equation*}
        \item Function $f(x, y)$ is strongly-convex-strongly-concave with constant $\mu$, i.e. for all $z_1, z_2 \in \mathcal{Z}$ it holds
	\begin{equation*}
		\langle F(z_1) - F(z_2), z_1 - z_2 \rangle \geq \mu \|z_1 - z_2\|^2.
	\end{equation*}
    \end{enumerate}
\end{assumption}

\noindent\textbf{Decentralized Communication}.

At each communication round, we use a graph to represent the connection between computing nodes. Denote the network of communications over time by the sequence of graphs $\{\mathcal{G}^k\}_{k=0}^\infty = \{(\mathcal{V}, \mathcal{E}^k)\}_{k=0}^{\infty}$, where $\mathcal{V} = \{1, \dots, M \}$ is the set of nodes and $\mathcal{E}^k$ is the set of available connections at $k$-th communication round. For each node $m \in \mathcal{V}$, we use the notation $\mathcal{N}_m^k = \{i \in \mathcal{V} | (i, m) \in \mathcal{E}^k \}$ to indicate the set of its neighbors at round $k$ and at that time it can only communicate with nodes in $\mathcal{N}_m^k$. 

\vspace{0.2cm}
\noindent\textbf{Gossip Matrices}
    Each computational node $m$ contains its own local vector $z_m = (x_m, y_m)$. It is required to satisfy the consensus constraints $z_1 = \dots = z_M$. For this purpose, we use a concept called \textit{gossip matrix}. 
    \begin{assumption} \label{assum:gossip}

        Each graph in the time-varying network correspond to a gossip matrix $W^k \in \mathbb R^{M \times M}$ that satisfies the following properties.
	\begin{enumerate}
		\item $W^k$ is positive semi-definite,
            \item $W^k_{i, j} = 0$ if $ i \neq j$ and $(i, j) \notin \mathcal{E}^k$,
            \item $\ker W^k = \spanop (\mathbf{1})$, where $\mathbf{1} = (1, \dots, 1) \in \mathbb R^M$.
	\end{enumerate}
    The number $\chi(W)= \frac{\lambda_{\max}(W) }{\lambda^+_{\min} (W) }$ is called the \textit{condition number} of gossip matrix $W$, where $\lambda_{\max}(W)$ and $\lambda^+_{\min}(W)$ denote the largest and smallest positive eigenvalue of $W$. For a time-varying network $\{\mathcal{G}^k\}_{k=1}^{\infty}$, the condition number is given by $\chi = \sup\limits_{k \in \mathbb N \cup \{0\}} \frac{\lambda_{\max}(W^k)}{\lambda_{\min}(W^k)}$.
    
    In this paper, we also consider network with Laplacian matrices $ \mathbf{L}(\mathcal{G}^k)$, which is a typical example of gossip matrix.
\end{assumption}

We introduce the \textit{consensus space}  $\mathcal{L} \subseteq \mathbb{R}^{Mn_z}$, defined by
\begin{equation}
	\mathcal{L} = \{ \mathbf{z} = (z_1^T, \dots, z_M^T)^T \in \mathbb{R}^{Mn_z}: z_1 = \dots = z_M\}.    
\end{equation}

Consider also the space $\mathcal{L}^{\perp} \subseteq \mathbb{R}^{Mn_z}$ which is an orthogonal complement to the space $\mathcal{L}$, defined by
\begin{equation}
	\mathcal{L}^{\perp} = \{ \mathbf{z} = (z_1^T, \dots, z_M^T)^T \in \mathbb{R}^{Mn_z}: \sum_{m = 1}^M z_m = 0\}.
\end{equation}

\section{Upper bounds}\label{sec:upper_bounds}

In this section, we cover two classes of time-varying networks: networks with connected skeleton and networks with small random Markovian changes. For both scenarios, we propose a decentralized optimization algorithm that uses an auxiliary consensus procedure. We show that for the considered classes of problems, the dependence of communication complexity on factor $\chi$ may be reduced from $\chi$ to $\sqrt\chi$ with additional terms. The overview of results is presented in Table~\ref{table:tw_upper_bounds} (the $O(\cdot)$ notation is omitted).

\begin{minipage}{15cm}
	\begin{table}[H]
		\begin{tabular}{|c|c|c|c|c|}
			\hline
			& arbitrary time-var. & slowly time-var. & connected skeleton & Markovian  \\ \hline
			comm. & $\chi\frac{L}{\mu}\log\frac{1}{\eps}$ & $\chi\frac{L}{\mu}\log\frac{1}{\eps}$  & $\sqrt\chi\log\chi\frac{L}{\mu}\log^2\frac{1}{\eps}$ & $\tau\left(\sqrt{\chi} + \frac{\rho^2}{(\lambda_{\min}^+)^2}\right)\frac{L}{\mu}\log^2\frac{1}{\eps}$ \\ \hline
			\rule{0.00cm}{0.5cm} oracle & $\frac{L}{\mu}\log\frac{1}{\eps}$ & $\frac{L}{\mu}\log\frac{1}{\eps}$ & $\frac{L}{\mu}\log\frac{1}{\eps}$ & $\frac{L}{\mu}\log\frac{1}{\eps}$ \\ \hline
			\rule{0.00cm}{0.4cm} paper & \cite{kovalev2022optimal} & \cite{kovalev2022optimal} & This paper & This paper \\ \hline
		\end{tabular}
		\caption{Upper bounds for saddle point problems over arbitrary and slowly time-varying graphs}\label{table:tw_upper_bounds}
	\end{table}
\end{minipage}

\vspace{0.2cm}
Our algorithms are based on extragradient method with consensus subroutine. They make several communication rounds after each extragradient step. After a sufficient number of communications, consensus is reached up to a desired accuracy. Such approximate averaging makes the trajectories of computational nodes almost synchronized. For each class of time-varying networks, we use a corresponding consensus subroutine and incorporate it into extragradient method to get a decentralized optimization algorithm.

\vspace{0.2cm}
\noindent\textbf{Accelerated Gossip with Connected Skeleton and Non-Recoverable Links}.

First, we focus on the type of graphs with connected skeleton. We assume that all graphs in the sequence have a common connected subgraph that we call a \textit{skeleton}. The edges may still appear and disappear, but each node remembers which incident links have failed at least once and stops communicating by that links. This strategy we call \textit{non-recoverable links}. Effectively the communication network only loses edges at each iteration, but remains connected. In other words, the graph of interest can be called "monotonically decreasing".

\begin{assumption} \label{assum:sleketon}
	Graph sequence $\{\mathcal{G}^k = (\mathcal{V}, \mathcal{E}^k)\}_{q = 0}^{\infty}$ has a connected skeleton: there exists a connected graph $\hat{\mathcal{G}} = (\mathcal{V}, \hat{\mathcal{E}})$ such that for all $k \in \mathbb{N} \cup \{ 0 \}$ we have $\hat{\mathcal{E}} \subset \mathcal{E}^k$, $\lambda_{max} (\mathbf{L}(\mathcal{G}^k)) \leq \lambda_{max}$ and $\lambda_{min}^+ \leq \lambda_{min}^+(\hat{\mathcal{G}})$.
\end{assumption}

	


    
    With these properties of the network, we introduce the following algorithm for consensus.

\begin{algorithm}
	\caption{Accelerated Gossip with Non-Recoverable Links (\texttt{AccGossipNonRecoverable})}\label{alg:AccNonRecoverable}
	\begin{algorithmic}
		\Require Vectors $z_1, \dots, z_M$, number of iterations $H$, current communication round number $k_0$. Step sizes $\eta, \beta > 0$.\\
		Construct column vector $\mathbf{z} = (z_1^T \dots z_M^T)^T$. \\
		Set $\mathbf{u}^0 = \mathbf{z}^0 = \mathbf{z}$. \\
		Every node $i = 1, 2, \dots, M$ initializes set of neighbors $\mathcal{N}_i = \mathcal{N}_i^{k_0}$.
		\For {$k = 0, 1, \dots, H - 1$}
		\For {$i = 1, 2, \dots, M$}
		
		Update the set of nodes to which the node communicates: $\mathcal{N}_i = \mathcal{N}_i \cap \mathcal{N}_i^{k_0 + k}$.
		
		$u_i^{k + 1} = z_i^k - \eta\left(|\mathcal{N}_i| z_i^k - \sum_{j \in \mathcal{N}_i} z_j^k\right) $
		
		$z_i^{k+1} = (1 + \beta)u_i^{k+1} - \beta u_i^k$
		\EndFor
		
		\EndFor \\
		\Return $z_1^H, \dots z_M^H$.
	\end{algorithmic}
\end{algorithm}

\begin{lemma} \label{lemma:concensus}
        \textbf{(From the proof of Theorem 4.3 in \cite{metelev2023consensus})}
        Let Assumptions \ref{assum:sleketon} hold and $\{ \hat{z}_m \}_{m=1}^M$ be output of Algorithm \ref{alg:AccNonRecoverable} with input $\{ z_m \}_{m=1}^M$ and step sizes $\eta = 1 / \lambda_{\max}$, $\beta = ( \sqrt{\chi} - 1) / ( \sqrt{\chi} + 1)$, where $\chi = \lambda_{\max} / \lambda_{\min}^+$. Then after $H$ iterations, it holds that
	\begin{equation} \label{eq:algo-output-est}
		\frac{1}{M} \sum_{m = 1}^M \| \hat{z}_m - \bar{{z}} \|^2 \leq \frac{2 \chi}{M} \left( 1 - \frac{1}{\sqrt{\chi}} \right)^H \sum_{m=1}^M \|z_m - \bar z\|^2,
	\end{equation}
	where $\bar z = \frac{1}{M} \sum_{m = 1}^M z_m$.
\end{lemma}

Based on Algorithm \ref{alg:AccNonRecoverable}, it is possible to develop a decentralized algorithm for solving the saddle point problem \ref{eq:main_prob} with a sequence of graphs that have a connected skeleton and non-recoverable links

\begin{algorithm}
	\caption{Time-Varying with Non-Recoverable Links Decentralized Extra Step Method}\label{alg:DESM-NonRecoverable}
	\begin{algorithmic}
		\Require Step size $\gamma > 0$; number of \texttt{AccGossipNonRecoverable} steps $H$, communication rounds $K$, number of iterations $N$.\\
		Choose $(x^0, y^0) = z^0 \in \mathcal{Z}, z_m^0 = z^0$.
		\For {$k = 0, 1, 2, \dots, N - 1$}
		
		Each machine $m$ computes  $\hat{z}_m^{k + 1/2} = z_m^k - \gamma \cdot F_m(z_m^k)$
		
		Communication: $\tilde{z}_1^{k + 1/2}, \dots, \tilde{z}_M^{k + 1/2} = \texttt{AccGossipNonRecoverable}(\hat{z}_1^{k + 1/2}, \dots, \hat{z}_M^{k + 1/2}, H)$
		
		Each machine $m$ computes $z_m^{k + 1/2} = \proj_{\mathcal{Z}}(\tilde{z}_m^{k + 1/2}$)
		
		Each machine $m$ computes  $\hat{z}_m^{k + 1} = z_m^k - \gamma \cdot F_m(z_m^{k + 1/2})$
		
		Communication: $\tilde{z}_1^{k + 1}, \dots, \tilde{z}_M^{k + 1} = \texttt{AccGossipNonRecoverable}(\hat{z}_1^{k + 1}, \dots, \hat{z}_M^{k + 1}, H)$
		
		Each machine $m$ computes $z_m^{k + 1} = \proj_{\mathcal{Z}}(\tilde{z}_m^{k + 1}$)    
		\EndFor
	\end{algorithmic}
\end{algorithm}

\begin{theorem} \label{th:eps-est}
	Suppose Assumptions \ref{assum:func-properties}, \ref{assum:sleketon} hold. Let problem \eqref{eq:main_prob} be solved by Algorithm \ref{alg:DESM-NonRecoverable} with $\gamma \leq \frac{1}{4 L_{\max}}$. Then in order to achieve $\varepsilon_0$-approximate consensus at each iteration, it takes 
	\[ H = \left( \sqrt{\chi} \log \left[ \chi \left( 4 + \frac{\frac{1}{2}\|z^0 - z^* \|^2 + \frac{Q^2}{2 L_{\max}^2}}{\varepsilon_0^2} \right) \right] \right)  \text{ communications,}\]
	where $Q^2 = \frac{1}{M} \sum_{m=1}^M \|F_m(z^*) \|^2$ and $z^*$ is a solution of \eqref{eq:main_prob}.
\end{theorem}

\begin{proof}
        \text
	Let us have $\eps_0$-accuracy of consensus after $k$ iterations, i.e.
	\[ \frac{1}{M} \sum_{m=1}^M \|z_m^k - z^k \|^2 \leq \varepsilon_0^2.\]

        We introduce the following notation:
        \begin{align*}
            g_m^k = F_m(z_m^k), \quad g_m^{k+1/2} = F_m(z_m^{k + 1/2}),
        \end{align*}
        and
        \begin{align*}
            z^k = \frac{1}{M} \sum_{m=1}^M z_m^k, \quad z^{k + 1/2} = \frac{1}{M} \sum_{m=1}^{M} z_m^{k+1/2}, \quad  g^k = \frac{1}{M} \sum_{m=1}^M g_m^k, \quad  g^{k+1/2} = \frac{1}{M} \sum_{m=1}^M g_m^{k+1/2}, \\
            \hat z^k = \frac{1}{M} \sum_{m=1}^M \hat z_m^k, \quad  \hat z^{k + 1/2} = \frac{1}{M} \sum_{m=1}^{M} \hat z_m^{k+1/2}, \quad \tilde z^k = \frac{1}{M} \sum_{m=1}^M \tilde z_m^k, \quad  \tilde z^{k + 1/2} = \frac{1}{M} \sum_{m=1}^{M} \tilde z_m^{k+1/2}.
        \end{align*}

        We have
        \begin{align*}
            \frac{1}{M} \sum_{m=1}^M \|g_m^k - g^k\|^2 \leq \frac{1}{M} \sum_{m=1}^M \|g_m^k\|^2.
        \end{align*}
        Let $T = 2 \chi \left( 1 - \frac{1}{\sqrt{\chi}} \right)^H$. Then
        \begin{align*}
            \frac{1}{M} \sum_{m=1}^M \|\tilde{z}_m^{k+1/2} - \tilde{z}^{k+1/2}\|^2 &\leq  \frac{T}{M} \sum_{m=1}^M \|\hat{z}_m^{k+1/2} - \hat{z}^{k+1/2}\|^2 = \frac{T}{M} \sum_{m=1}^M \|z_m^k -\gamma g_m^k - z^k -\gamma g^k\|^2 \\
            & \leq \frac{2T}{M} \sum_{m=1}^M \|z_m^k - z^k \|^2 + \frac{2 T \gamma^2}{M} \sum_{m=1}^M \|g_m^k - g^k \|^2 \\
            & \leq \frac{2T}{M} \sum_{m=1}^M \|z_m^k - z^k \|^2 + \frac{2 T \gamma^2}{M} \sum_{m=1}^M \|g_m^k\|^2 = 2T \varepsilon_0^2 + \frac{2 T \gamma^2}{M} \sum_{m=1}^M \|g_m^k\|^2.
        \end{align*}
        On the other side
        \begin{align*}
            \frac{1}{M} \sum_{m=1}^M \| g_m^k \|^2 & = \frac{1}{M} \sum_{m=1}^M \| F_m(z_m^k)\|^2 \leq \frac{2}{M} \sum_{m=1}^M \| F_m(z_m^k) - F_m(z^*)\|^2 + \frac{2}{M} \sum_{m=1}^M \|F_m(z^*)\|^2 \\
            & \leq 2 L_{\max}^2 \|z_m^k - z^*\|^2 + \frac{2}{M} \sum_{m=1}^M \|F_m(z^*)\|^2 
            \leq 2 L_{\max}^2 \|z_m^0 - z^*\|^2 + \frac{2}{M} \sum_{m=1}^M \|F_m(z^*)\|^2.
        \end{align*}
        Let $Q^2 = \frac{1}{M} \sum_{m=1}^M \|F_m(z^*) \|^2$, then we have
        \begin{align*}
            \frac{1}{M} \sum_{m=1}^M \|z_m^{k+1/2} - z^{k+1/2}\|^2 
            &= \frac{1}{M} \sum_{m=1}^M \|\proj_{\mathcal Z} \tilde {z}_m^{k+1/2} - \proj_{\mathcal Z} \tilde{z}^{k+1/2}\|^2 \leq \frac{1}{M} \sum_{m=1}^M \|\tilde{z}_m^{k+1/2} - \tilde{z}^{k+1/2}\|^2 \\
            &  \leq 2T \left(\varepsilon_0^2 + 2 {\gamma}^2\left( L_{\max}^2 \|z_m^0 - z^*\|^2 + Q^2 \right) \right) \leq 2T \left(\varepsilon_0^2 + \frac{1}{8} \|z_m^0 - z^*\|^2 + \frac{Q^2}{8 L_{\max}^2}\right)\\
            & = \chi \left( 1 - \frac{1}{\sqrt{\chi}} \right)^H \left(4 \varepsilon_0^2 + \frac{1}{2} \|z^0 - z^*\|^2 + \frac{Q^2}{2 L_{\max}^2} \right).
        \end{align*}
	If we take 
        \begin{align*}
        H \geq \sqrt{\chi} \log \left[\chi \left( 4 + \frac{\frac{1}{2}\|z^0 - z^* \|^2 + \frac{Q^2}{2 L_{\max}^2}}{\varepsilon_0^2} \right) \right], 
        \end{align*}
        then
        \begin{align*}
            \frac{1}{M} \sum_{m = 1}^M\|z_m^{k+1/2} - z^{k+1/2}\|^2 \leq \varepsilon_0^2.
        \end{align*}
	Analogically, we get the same estimate for $H$ to ensure that 
	\begin{align*}
  \frac{1}{M} \sum_{m = 1}^M \| z_m^{k+1} - z^{k+1} \|^2 \leq \varepsilon_0^2.
	\end{align*}
	Hence, to achieve accuracy $\varepsilon_0$, we need to perform 
	
	\[ H = \mathcal{O} \left( \sqrt{\chi} \log \left[ \chi \left( 4 + \frac{\frac{1}{2}\|z^0 - z^* \|^2 + \frac{Q^2}{2 L_{\max}^2}}{\varepsilon_0^2} \right) \right] \right) \text{ communications.}\]
\end{proof}

\begin{theorem} \label{th:DESM-est}
    \textbf{(From Theorem 6 in \cite{beznosikov2020distributed})}
	Let $\{z_m^k\}_{k \geq 0}$ denote the iterates of Algorithm \ref{alg:DESM-NonRecoverable} for solving problem \ref{eq:main_prob} . Let Assumptions \ref{assum:func-properties}, \ref{assum:sleketon} be satisfied. Then if $\gamma \leq \frac{1}{4L_{\max}}$ , we have the following estimates
	\[ \| \bar{z}^{N} - z^* \| = \mathcal{O} \left( \|z^0 - z^*\|^2 \exp \left( - \frac{\mu K}{8L \cdot H} \right) \right). \]
\end{theorem}

\begin{corollary}
	In the setting of Theorem \ref{th:eps-est} and Theorem \ref{th:DESM-est}, if $H = \mathcal{O} \left(  \sqrt{\chi} \log\chi \log(1/\eps) \right)$, then the number of communication rounds required for Algorithm \ref{alg:DESM-NonRecoverable} to obtain $\varepsilon$-solution is upper bounded by 
	\[ \mathcal{O} \left( \sqrt{\chi} \log\chi \frac{L}{\mu} \log^2\frac{1}{\varepsilon} \right), \] and the number of local computations on each device is upper bounded by
	\[ \mathcal{O} \left(  \frac{L}{\mu} \log \frac{1}{\varepsilon} \right). \]
\end{corollary}

\vspace{0.2cm}
\noindent\textbf{Consensus for networks with Markovian changes}

This subsection is devoted to slowly time-varying graphs with random changes satisfying Markovian law. At each iteration, several randomly chosen edges may appear or disappear. The choice of edges depends only on the current graph topology, so the sequence of graphs is a Markovian process.

Let introduce some requirements for time-varying network with Markovian changes.

\begin{assumption} \label{assum:markov-chain}
    The communication network satisfy the following conditions
    \begin{enumerate}
        \item $\{ W^k \}_{k=0}^\infty$ is a stationary Markov chain on $(W_G, W_\sigma)$, where $W_G$ is the set of all possible gossip matrices for the network and $W_{\sigma}$ is the $\sigma$-field on $W_G$ and the chain $\{ W^k \}_{k=0}^\infty$ has a Markov kernel $Q$ and a unique stationary distribution $\pi$.
        \item $Q$ is uniformly geometrically ergodic with mixing time $\tau \in \mathbb{N}$, i.e., for every $m \in \mathbb N$, 
	\[ \Delta (Q^m) = \sup_{W, W' \in W_G} (1/2) \| Q^m(W, \cdot) - Q^m(W', \cdot) \|_{TV} \leq    
        (1/4)^{\lfloor m / \tau \rfloor}.\]
        \item For all $k \in \mathbb N \cup \{0\}$, it holds $\mathbb E_\pi [ W(q) ] = \bar W$ and the $\bar W$ satisfies Assumption \ref{assum:gossip}. 
        
        Denote $\lambda_{\max} = \lambda_{\max}(\bar W)$, $\lambda_{\min}^+ = \lambda_{\min}^+(\bar W)$, $\chi = \frac{\lambda_{\max}}{\lambda_{\min}}$. 
        \item For any graph $\mathcal{G}$ that can appear in the network it holds:
	\[ \| W(\mathcal{G}) - \bar{W} \| \leq \rho. \]
    \end{enumerate}
\end{assumption}

Consider the consensus search problem:

\begin{equation} \label{eq:markov-consensus}
    \begin{aligned}
        \min_{\mathbf{z} \in \mathbb{R}^{M n_z}} & \left[ r(\mathbf{z}) = \left\| \left(\sqrt{\bar W } \otimes \mathbf{I}_{n_z}\right) \mathbf{z} \right\|^2 \right] \\
        \textrm{s.t.} & \sum_{m = 1}^M z_m = \sum_{m=1}^M z_m^0 
    \end{aligned},
\end{equation}
where $\mathbf{z} = (z_1^T, \dots, z_M^T)^T$.

\begin{algorithm} 
	\caption{Accelerated consensus over graphs with Markovian changes (\texttt{ACOGWMC})}\label{alg:ACOGWMC}
	\begin{algorithmic}
		\Require stepsize $\gamma > 0$, momentums $\theta, \eta, \beta, p$, number of iterations $N$, batchsize limit $S$.\\
		Choose $z_f^0 = z^0, T^0 = 0$, set the same random seed for generating $\{J_k\}$ on all devices
		\For {$k = 0, 1, \dots, N - 1$}
		
		$z_g^k = \theta z_f^k + (1 - \theta) z^k$
		
		Sample $J_k \sim \text{Geom}(1/2)$
		
		Send $z_g^k$ to neighbors in the networks $\{ \mathcal{G}^{T^k+i} \}_{i=1}^{2^{J_k} B}$
		
		Compute $g^k = g_0^k + 
		\begin{cases}
			2^{J_k} \left( g_{J_k}^k - g_{J_k - 1} ^ k \right), & \text{if} \; 2^{J_k} \leq S \\
			0, & \text{otherwise}
		\end{cases}
		$
		
		with $g_j^k = 2^{-j}B^{-1}\sum_{i = 1}^{2^j B} W^{T^k+i} z_g^k$
		
		$z_f^{k+1} = z_g^k - p \gamma g^k$
		
		$z^{k+1} = \eta z_f^{k+1} + (p - \eta)z_f^k + (1-p)(1-\beta)z^k + (1-p)\beta z_g^k$
		
		$T^{k+1} = T^k + 2^{J_k} B$
		\EndFor
	\end{algorithmic}
\end{algorithm}

\begin{theorem} \label{th:markov-consensus-est}
    \textbf{(Theorem 1 from \cite{metelev2023decentralized})}
	Let Assumptions \ref{assum:markov-chain} hold. Let problem \eqref{eq:markov-consensus} be solved by Algorithm \ref{alg:ACOGWMC}. Then for any $b \in \mathbb{N}$, \[\gamma \in \left( 0; \min \left\{ \frac{3}{4 \lambda_{\max}} ; \frac{\lambda_{\min}^3}{[1800 \rho^2(\tau b^{-1} + \tau^2 b^{-2})]^2} \right\}\right),\] and $\beta, \theta, \eta, p, S, B$ satisfying \[ p = \frac{1}{4}, \, \beta = \sqrt{\frac{4p^2 \mu \gamma}{3}}, \, \eta = \frac{3\beta}{p \mu \gamma} = \sqrt{\frac{12}{\mu \gamma}}, \, \theta = \frac{p\eta^{-1} - 1}{\beta p \eta^{-1} - 1}\]
	\[ S = \max \left\{2; \sqrt{\frac{1}{4} \left(1 + \frac{2}{\beta} \right)}\right\}, \, B = \lceil b \,\log_2 S \rceil, \]
	it holds that
	\begin{align*}
		&\mathbb{E} \left[ \|z^N - z^*\|^2 + \frac{24}{\lambda_{\min}} (r(z_f^N) - r(z^*))\right] \\ &\quad= \mathcal{O} \left( \exp \left( -N \sqrt{\frac{p^2 \lambda_{\min} \gamma}{3}} \right) \left[\|z^0 - z^*\|^2 + \frac{24}{\lambda_{\min}} (r(z^0) - r(z^*)) \right]  \right),
	\end{align*}
	where $z^*_{m} = \frac{1}{M} \sum_{i=1}^M z_{i}$ for $m = 1, \dots, M$.
\end{theorem}

\begin{corollary}
	In the setting of Theorem \ref{th:markov-consensus-est}, if $b = \tau$ and $\gamma \simeq \text{\normalfont{min}} \left\{ \frac{1}{\lambda_{\text{\normalfont{max}}}}; \frac{\lambda_{\text{\normalfont{min}}}^3}{\rho^4} \right\}$, then in order to achieve $\varepsilon$-approximate solution (in terms of $\mathbb{E} [\|z - z^*\|^2] \lesssim \varepsilon$) it takes
	\[ \tilde{\mathcal{O}} \left( \tau \left[ \sqrt{\chi} + \frac{\rho^2}{\lambda_{\text{\normalfont{min}}}^2} \log \frac{1}{\varepsilon} \right] \right) \text{ communications.}\]
\end{corollary}

\begin{algorithm}
	\caption{Time-Varying Decentralized Extra Step Method with Markovian changes}\label{alg:DESM-markov}
	\begin{algorithmic}
		\Require Step size $\gamma \leq \frac{1}{4L}$, number of \texttt{AccGossipNonRecoverable} steps $H$, number of iterations $N$.\\
		Choose $(x^0, y^0) = z^0 \in \mathcal{Z}, z_m^0 = z^0$.
		\For {$k = 0, 1, 2, \dots, N$}
		
		Each machine $m$ computes  $\hat{z}_m^{k + 1/2} = z_m^k - \gamma \cdot F_m(z_m^k)$
		
		Communication: $\tilde{z}_1^{k + 1/2}, \dots, \tilde{z}_M^{k + 1/2} = \texttt{ACOGWMC}
		(\hat{z}_1^{k + 1/2}, \dots, \hat{z}_M^{k + 1/2}, H)$
		
		Each machine $m$ computes $z_m^{k + 1/2} = \text{proj}_{\mathcal{Z}}(\tilde{z}_m^{k + 1/2}$)
		
		Each machine $m$ computes  $\hat{z}_m^{k + 1} = z_m^k - \gamma \cdot F_m(z_m^{k + 1})$
		
		Communication: $\tilde{z}_1^{k + 1}, \dots, \tilde{z}_M^{k + 1} = \texttt{ACOGWMC}
		(\hat{z}_1^{k + 1}, \dots, \hat{z}_M^{k + 1}, H)$
		
		Each machine $m$ computes $z_m^{k + 1} = \text{proj}_{\mathcal{Z}}(\tilde{z}_m^{k + 1}$)    
		\EndFor
	\end{algorithmic}
\end{algorithm}

\begin{theorem}
	Let Assumptions \ref{assum:func-properties}, \ref{assum:markov-chain} hold. Let problem \eqref{eq:main_prob} be solved by Algorithm \ref{alg:DESM-markov}. Then, if $\gamma \leq \frac{1}{4 L_{\max}}$ and $H = \mathcal{O}  \left( \tau \left[ \sqrt{\chi} + \frac{\rho^2}{\lambda_{\text{\normalfont{min}}}^2} \text{\normalfont{log}} \frac{1}{\varepsilon} \right] \right)$ it holds that to achieve $\varepsilon$-solution (in terms of $\mathbb{E} [f(z) - f(z^*)] \lesssim \varepsilon$) it takes
	\[ \tilde{\mathcal{O}} \left( \tau \left[ \sqrt{\chi} + \frac{\rho^2}{\lambda_{\text{\normalfont{min}}}^2} \right] \frac{L}{\mu} \; \text{log}^2\frac{1}{\varepsilon} \right) \textit{communications and}\]
	\[ \mathcal{O} \left(  \frac{L}{\mu} \; \text{log}\frac{1}{\varepsilon} \right) \textit{local computations on each node.} \]
\end{theorem}

\section{Lower bounds}\label{sec:lower_bounds}

The previous section was devoted to applying new consensus algorithms for specific types of time-varying networks: networks with connected skeleton and graphs changing according to Markovian law. In this section, we show that without specific constraints on the network change, such as connected skeleton or Markovian changes, acceleration is not obtained on generalised slowly changing graphs, even if the constraints on the rate of change are very high. That is, under constant network constraints, when at most a constant number of edges in the network per iteration is changed, the worst-case dependence on $\chi$ cannot be improved in comparison to arbitrary changing networks. In particular, we show that it sufficient to change only two edges per iteration.

We start with the definition of black-box procedure class of algorithms on which we will evaluate the lower bound.

\begin{definition}
    An algorithm with $T$ local iterations and $K$ communication rounds that satisfies following properties is called a \textit{black-box procedure}, denoted by $\textbf{BBP}(T, K)$.
    
    Each node $m$ maintains a local memory with $\mathcal{M}^x_m$ and $\mathcal{M}^y_m$ for the $x$- and $y$-variables, which are initialized as $\mathcal{M}^x_m = \mathcal{M}^y_m = \{0\}$. $\mathcal{M}^x_m$ and $\mathcal{M}^y_m$ are updated as follows:
    \begin{itemize} 
    \item \textbf{Local computation:} Each node $m$ computes and adds to its $M^x_m$ and $M^y_m$ a finite number of points $x$, $y$, each satisfying
    \[x \in \spanop \{ x', \nabla_x f_m(x'', y'')\}, \quad y \in \spanop \{ y', \nabla_y f_m(x'', y'')\},\]
    for given $x', x'' \in \mathcal{M}^x_m$ and $y', y'' \in \mathcal{M}^y_m$.
    \item \textbf{Communication:} $M^x_m$ and $M^y_m$ are updated according to
    \[ \mathcal{M}^x_m := \spanop \left\{ \bigcup\limits_{(i, m) \in \mathcal{E}^k} \mathcal{M}^x_i\right\}, \quad \mathcal{M}^y_m := \spanop \left\{ \bigcup\limits_{(i, m) \in \mathcal{E}^k} \mathcal{M}^y_i \right\}, \]
    where $\mathcal{G}^k = (\mathcal{V}, \mathcal{E}^k)$ is the current state of network.
    \item \textbf{Output:} The final global output at current moment of time is calculated as:
    \[\hat x \in \spanop \left\{ \bigcup\limits_{m = 1}^M \mathcal{M}^x_m\right\}, \quad \hat y \in \spanop \left\{ \bigcup\limits_{m = 1}^M \mathcal{M}^y_m\right\}. \]
    \end{itemize}
    
\end{definition}

To estimate the lower bound of distributed saddle point problem \eqref{eq:main_prob}, we need to provide a "bad function" and a "bad sequence of graphs" such that any black-box procedure cannot solve it using less than a given number of rounds. Using the time-varying network from \cite{metelev2023decentralized} and the objective function from \cite{zhang2019lower} we can prove the following theorem.

\begin{theorem}
	For any $L > \mu > 0$ and any $\chi \geq 1$, there exists a decentralized distributed saddle point problem on $\mathcal{X} \times \mathcal{Y} = \mathbb{R}^n \times \mathbb{R}^n$ (where $n$ is sufficiently large) with $x^*, y^* \ne 0$, a sequence of graphs $\{\mathcal{G}^k = (\mathcal{V}, \mathcal{E}^k)\}_{k = 0}^{\infty}$, where consecutive graphs differ in no more than two edges, and a sequence of corresponding gossip matrices $\{W^k\}_{k = 0}^{\infty}$ with characteristic number $\chi$, such that for any output $\hat{x}, \hat{y}$ after $K$ communication rounds of any \textbf{BBP}, the following estimate hold:
	\[\|\hat{x} - x^*\|^2 + \|\hat{y} - y^*\|^2 \geq \Omega \left(\text{exp}\left(-\frac{32 \mu}{L - \mu} \cdot \frac{K}{\chi} \right) \| y_0 - y^* \| ^ 2\right) \]
\end{theorem}

\begin{proof}
	Let us introduce the graph $T_{a, b}$ from \cite{metelev2023decentralized}. This graph contains two partitions: left and right, these partitions consist of $a$ and $b$ vertices respectively. Each vertex in partitions is connected to its root. These roots are connected to the another vertex called  central root.
 
     Consider the network with $|\mathcal{V}| = 2 d + 3$ nodes ($d \geq 2$). We select a node to be the central root and two other nodes to be the left and right roots. Central root can be changed over time but partition roots are fixed. We also select $\left[\frac{d}{2}\right]$ fixed vertices for each partition and denote them by $\mathcal{V}_1$ (left side) and $\mathcal{V}_2$ (right side). At each communication round $k$, the graph $\mathcal{G}^k$ has the form $T_{a, b}$, where $a + b = 2d$ and $a, b \geq \left[\frac{d}{2}\right]$.
     
     The communication network changes alternatively in two phases. The first phase starts with graph $T_{2 d - \left[\frac{d}{2} \right], \left[\frac{d}{2} \right]}$ and ends with graph $T_{\left[\frac{d}{2} \right], 2 d - \left[\frac{d}{2} \right]}$. 
     At each iteration, the central root is moved to the right partition and one vertex from left partition, but not in $\mathcal{V}_1$, becomes the central root. The second phase has the same procedure, but from right to left.

	We modify the objective function from section B.1 in \cite{beznosikov2020distributed} based on our graph type:
	
	\begin{equation} \label{eq:node-func}
		f_m(x, y) = 
		\begin{cases}
			f_1(x,y) = \frac{M}{2 | \mathcal{V}_2|} \cdot \frac{L}{2} x^T A_1 y + \frac{\mu}{2} \|x\|^2 - \frac{\mu}{2} \|y\|^2 + \frac{M}{2|\mathcal{V}_2|} \cdot \frac{L^2}{2\mu} e_1^Ty, & m \in \mathcal{V}_2 \\
			f_2(x,y) = \frac{M}{2|\mathcal{V}_1|} \cdot \frac{L}{2} x^T A_2 y + \frac{\mu}{2} \|x\|^2 - \frac{\mu}{2} \|y\|^2, & m \in \mathcal{V}_1\\
			f_3(x,y) = \frac{\mu}{2} \|x\|^2 - \frac{\mu}{2} \|y\|^2, & \text{otherwise}
		\end{cases}\,.
	\end{equation}
	where $e_1 = (1, 0, \dots , 0)$ and
	
	\begin{center}
		$A_1 =$ 
		$\begin{pmatrix}
			& 1 & 0 &  & & & & & & \\ 
			& & 1 & -2 & & & & & & \\
			& & & 1 & 0 & & & & & \\
			& & & & 1 & -2 & & & & \\
			& & & & & \dots & \dots & & & \\
			& & & & & & 1 & -2 & & \\
			& & & & & & & 1 & 0 & \\
			& & & & & & & & 1 &
		\end{pmatrix}$, \quad
		$A_2 = $
		$\begin{pmatrix}
			& 1 & -2 &  & & & & & & \\ 
			& & 1 & 0 & & & & & & \\
			& & & 1 & -2 & & & & & \\
			& & & & 1 & 0 & & & & \\
			& & & & & \dots & \dots & & & \\
			& & & & & & 1 & 0 & & \\
			& & & & & & & 1 & -2 & \\
			& & & & & & & & 1 &
		\end{pmatrix}$.
	\end{center}

	Consider the problem with global objective function:
	\[
	f(x,y) := \frac{1}{M} \sum_{m=1}^{M} f_m(x, y) = \frac{1}{M}(|\mathcal{V}_2| \cdot f_1(x, y) + |\mathcal{V}_1| \cdot f_2(x, y) + (M - |\mathcal{V}_1| - |\mathcal{V}_2|) \cdot f_3(x, y)
	\]
	\begin{equation} \label{eq:global-func}
		= \frac{L}{2} x^T A y + \frac{\mu}{2} \|x\|^2 - \frac{\mu}{2} \|y\|^2 + \frac{L^2}{4\mu} e_1^Ty, \; \text{with} \;A = \frac{1}{2}(A_1 + A_2).
	\end{equation}

        We estimate the number of communication rounds required to obtain a new non-zero element in a local memory using the following lemma.
	
	\begin{lemma}
		Let Problem \eqref{eq:main_prob} be solved by any \textbf{BBP}. Then after $K$ communication rounds only the first $\lfloor \frac{K}{d} \rfloor$ coordinates of the global output can be non-zero while the rest of the $n - \lfloor \frac{K}{d} \rfloor$ coordinates are strictly equal to zero.
	\end{lemma}
	\begin{proof}
  
		Between two consecutive communication rounds, only nodes from $\mathcal{V}_1$ and $\mathcal{V}_2$ can add new non-zero coordinate to their local memory, but in this interval two nodes from $\mathcal{V}_1$ and $\mathcal{V}_2$ cannot progress simultaneously. Moreover, at most one new non-zero coordinate can be added between two communication rounds (see \cite{beznosikov2020distributed} for more details).
  
		Hence, we constantly have to transfer information from the group of nodes $\mathcal{V}_1$ to $\mathcal{V}_2$ and back to get new non-zero coordinates. For each new non-zero coordinate, we need at least one local computation. So $T > K$ is required. We know from \cite{metelev2023decentralized} that each transfer requires at least $d$ communication rounds. So after $K$ communication rounds at most first $\lfloor \frac{K}{d} \rfloor$ coordinates of the global output can be non-zero.  
	\end{proof}

        We use the following auxiliary lemmas to estimate lower bound on convergence.
	\begin{lemma} 
		\textbf{(Lemma 4 from \cite{beznosikov2020distributed})}
		Let $\alpha = \frac{4\mu^2}{L^2}$ and $q = \frac{1}{2}(2 + \alpha - \sqrt{\alpha^2 + 4\alpha}) \in (0;1)$
		- the smallest root of $q^2 - (2 + \alpha)q + 1 = 0$, and let introduce approximation $\Bar{y}^*$
		\[\Bar{y}_i^* = \frac{q^i}{1-q}.\]
		Then error between approximation and real solution of \eqref{eq:global-func} can be bounded 
		\[\| \Bar{y}^* - y^* \| \leq \frac{q^{n+1}}{\alpha(1-q)}.\]
	\end{lemma}

	\begin{lemma}
		\textbf{(Lemma 5 from \cite{beznosikov2020distributed})}
		Consider a distributed saddle point problem in form \eqref{eq:node-func}, \eqref{eq:global-func} with sequence of graphs $\{\mathcal{G}_k = (\mathcal{V}, \mathcal{E}_k)\}_{k = 0}^{\infty}$ and sequence of corresponding gossip matrices $\{W(\mathcal{G}_k)\}_{k = 0}^{\infty}$. For any pairs $T, K (T > K)$ one can found size of the problem $n \geq \text{max} \{2 \text{log}_q \left( \frac{\alpha}{4 \sqrt{2}} \right), 2K \}$, where $\alpha = \frac{4\mu^2}{L^2}$ and $q = \frac{1}{2} (2 + \alpha - \sqrt{\alpha^2 + 4 \alpha}) \in (0; 1)$. Then, any output $\hat{x}, \hat{y}$ produced by any $\mathbf{BBP}(T, K)$ after $K$ communication rounds and $T$ local computations, is such that
		\[\|\hat{x} - x^*\|^2 + \|\hat{y} - y^*\|^2  \geq q^{\frac{2K}{d}} \frac{\|y_0 - y^* \|^2}{16}.\]
	\end{lemma}
	
	Using result from the proof of Proposition 3.6 in \cite{zhang2019lower}, we have:
	
	\[\text{ln}(q) \geq \frac{q -1}{q} = \frac{2}{1 - \sqrt{\frac{L^2}{\mu^2} + 1}} \geq \frac{2}{1 - \frac{L}{\mu}} = \frac{-2 \mu}{L - \mu}.\]
	
	For each graph in our sequence we map a weighted Laplacian from Lemma 8 in \cite{metelev2023decentralized}, so $\chi \leq 8d$.
 
        Hence
	\[ \text{ln} (q) \cdot \frac{2K}{d} \geq \frac{-4 \mu}{L - \mu} \cdot \frac{K}{d} \geq \frac{-32 \mu}{L - \mu} \cdot \frac{K}{\chi}.\]
	
	We get
	\[q^{\frac{2K}{d}} = \text{exp} \left( \text{ln} (q) \cdot \frac{2K}{d} \right) \geq \text{exp} \left( \frac{-32 \mu}{L - \mu} \cdot \frac{K}{\chi} \right).\]
	
	So we obtain
	\[\|\hat{x} - x^*\|^2 + \|\hat{y} - y^*\|^2 = \Omega \left(\text{exp}\left(-\frac{32 \mu}{L - \mu} \cdot \frac{K}{\chi} \right) \| y_0 - y^* \| ^ 2\right). \]
\end{proof}

\begin{corollary}
	In the setting of Theorem 1, the number of communication rounds required to obtain a $\varepsilon$-solution is lower bounded by
	\[\Omega \left( \chi \frac{L}{\mu} \cdot \text{log} \left( \frac{\| y^* \| ^ 2}{\varepsilon} \right) \right),\]
	and the number of local calculations on each node is lower bounded by:
	\[\Omega \left( \frac{L}{\mu} \cdot \text{log} \left( \frac{\| y^* \| ^ 2}{\varepsilon} \right) \right).\]
\end{corollary}

\section{Conclusion}

In this paper, we studied min-max optimization over slowly time-varying graphs. We showed that if the graph changes in an adversarial manner and a constant number of edges is changed at each iteration, then lower complexity bounds coincide with those for arbitrary changing networks. Moreover, we showed that for particular time-varying graphs -- networks with connected skeleton and networks with Markovian changes -- acceleration of communication procedures is possible. We proposed the corresponding algorithms for saddle point problems for these two classes of problems.

The research was supported by Russian Science Foundation (project No. 23-11-00229), \\ \url{https://rscf.ru/en/project/23-11-00229/}.

\bibliographystyle{abbrv}
\bibliography{references}

\end{document}